\documentclass[10pt]{amsart}

\usepackage[margin=1.5in]{geometry}               

\usepackage{enumerate}
\usepackage{amssymb}
\usepackage{amsmath}
\usepackage{amscd}
\usepackage{amsthm}
\usepackage{amsfonts}
\usepackage{graphicx}
\usepackage{verbatim}
\usepackage{hyperref}
\usepackage{epstopdf}
\usepackage[utf8]{inputenc}
\usepackage{bbm}
\usepackage{tikz} 
\usepackage{mathtools} 

\usepackage{float}

\newtheorem{theorem}{Theorem}

\newtheorem{lemma}[theorem]{Lemma}

\theoremstyle{definition}

\newtheorem{corollary}[theorem]{Corollary}
\theoremstyle{definition}

\theoremstyle{definition}

\theoremstyle{definition}
\theoremstyle{definition}\newtheorem{remark}[theorem]{Remark}
\theoremstyle{definition}
\numberwithin{theorem}{section}

\newcommand{\cA}{\mathcal{A}}
\newcommand{\cB}{\mathcal{B}}

\newcommand{\R}{\mathbb{R}}

\newcommand{\C}{\mathbb{C}}

\newcommand{\onto}{\xymatrix{\ar@{>>}[r]&}}
\newcommand{\da}[4]{\xymatrix{#1 \ar@<.5ex>[r]^{#2} \ar@<-.5ex>[r]_{#3} & #4}}

\newif\ifdraft\drafttrue

\def\BState{\State\hskip-\ALG@thistlm}
\newfloat{algorithm}{t}{lop}

\author{Rene Rühr}
\title[Unique Ergodicity of IETs]{A Convexity Criterion for Unique Ergodicity of Interval Exchange Transformations}
\date{}

\begin{document}
\maketitle

There is a meta-conjecture in metric number theory that states that any Diophantine property that holds for generic vectors in $\R^n$ should hold for generic vectors on nondegenerate subvarities, see Kleinbock's survey \cite{kleinbock69some}[Section 4]. Mahler asked for example
whether all most points on the curve $s\mapsto (s,s^2,\dots,s^d)$ are very well approximable. This has been answered affirmatively by Sprindžuk. It is believed that an analogue of this phenomenon holds also for the unique ergodicity property for interval exchange transformations. Minsky and Weiss \cite{minskyweiss} provide a general condition for unique ergodicity to hold. In this note we provide an easy-to-check criterion for their condition to be satisfied.

Let $\sigma$ denote a permutation of $d$ elements. 
Let $\Omega$ denote the antisymmetric matrix
\[
\Omega_{ij} = 
\left\{
\begin{matrix}
1 && i>j, \, \sigma(i)<\sigma(j) \\ 
-1 && i<j, \, \sigma(i) >\sigma(j) \\ 
0 && \text{otherwise.}
\end{matrix} \right.
\]
Let ${\bf a}=(a_1,\dots,a_d)\in \R^d_+$ be a row vector with positive entries $a_i>0$
and the associated interval $I_{\bf a}=[0,\sum a_i)$, 
which is divided into $d$ subintervals $I_i=[x_{i-1},x_i)$ where $x_i=\sum_{j\leq i} a_j$ are called discontinuities. 
Also introduce $x'_i=\sum_{j\leq i} a_{\sigma^{-1}(j)}$.
An \textit{interval exchange transformation} $T:I_{\bf a}\to I_{\bf a}$ defined by the data $(\sigma,{\bf a})$ is the map 
\[
T(x)=x+({\bf a}\Omega)_j=x-x_j+x_{\sigma(j)}' \quad \text{ for }  x\in I_j.
\]
In words, $T$ permutes the intervals $I_j$ of length $a_j$ according to $\sigma$. 
The form $\Omega_{ij}$ captures the exchange of two intervals $I_i, I_j$ relative to each other.

We shall always assume that the permutation $\sigma$ is irreducible 
in the sense that if $\{1,\dots,k\}\subset\cA=\{1,\dots,d\}$ is invariant under $\sigma$ then $k=d$.

Masur \cite{masur} and Veech \cite{veechgauss} proved independently that for almost all $\bf{a}\in \R^d_+$, 
the interval exchange transformation $T$ associated to $(\sigma,\bf{a})$ is uniquely ergodic, 
that is, the only $T$-invariant probability measure on $I$ is Lebesgue measure. 

Motivated by a conjecture of Mahler in the theory of Diophantine approximation, 
Minsky and Weiss \cite{minskyweiss} proved the following theorem.

\begin{theorem}[Minsky-Weiss]
\label{thm:mw-mahler}
	Let $\sigma=(d,\dots,1)$ and $\textbf{a}(s)=(s,s^2,\dots,s^d)$. 
	Then for Lebesgue almost all $s>0$, the interval exchange transformation 
	associated to $(\sigma,\textbf{a}(s))$ is uniquely ergodic.
\end{theorem}

In this note, we wish to note how to extend the theorem of Minsky-Weiss 
to arbitrary permutations by means of a simple convexity criterion.

We first recall the theorem from which Theorem~\ref{thm:mw-mahler} is deduced, 
which requires us to introduce more definitions.
A connection of $T$ is a triple $(m,x_i,x_j)$ for which $T^m(x_i)=x_j$. 
As noted by Keane, if the coordinates $a_k$ of $\textbf{a}$
are rationally independent then $T$ has no connections. 
We shall restrict to curves ${\bf a}(s)$ for which this is the case for almost all $s$. 
This implies that there are no $T$-invariant atomic probability measures.

Let ${\bf b}=(b_1,\dots,b_d)\in\R^d$ be a row vector, which we will take to be ${\bf b} = \dot{\mathbf{a}}$, the derivative of a curve ${\bf a}(s)$.
Define $y_i=\sum_{j\leq i} b_i$ and $y_i'=\sum_{j\leq i}b_{\sigma^{-1}(j)}$. We put
\[
	L(x)=(\Omega{\bf b}^T)_i=y_i-y'_{\sigma(i)} \quad \text{ for }  x\in I_i.
\]
We call $({\bf a},{\bf b})\in \R^d_+\times \R^d$ a positive pair if $\mu(L)>0$ for any $T$-invariant probability measure $\mu$.
The following is a simplified statement of Theorem~6.2 in \cite{minskyweiss}.
\begin{theorem}[Minsky-Weiss]
\label{thm:mw-curves}
	If $\mathbf{a}:A\to\R^d_+$ is a $C^2$-curve defined on an interval $A\subset\R$ and $\sigma$ a permutation for which
	$(\mathbf{a}(s),\dot{\mathbf{a}}(s))$ is positive for Lebesgue almost all $s\in A$
	then $T$ associated to $(\sigma,\mathbf{a}(s))$ is uniquely ergodic for almost all $s\in A$.
\end{theorem}

While the the condition seems hard to check 
- involving all $T$-invariant $\mu$'s  (from which want to deduce that there is only one!) - 
we see however that a sufficient criterion for positivity is $L(x)>0$ pointwise for every $x\in I$.

Let us now explain a suspension construction of Masur for an interval exchange transformation $T$ associated to $(\sigma,{\bf a})$.
Let ${\bf b}\in\R^d$ be a ``height vector'' (associated to the ``length vector'' ${\bf a}$) 
and define $\zeta_i$ to be $(a_i,b_i)\in \R^2$ and their slopes to be $\kappa_i=\frac{b_i}{a_i}$.

Let $\Gamma_t$ be the curve obtained by connecting the points
\[
C_0=C_0^t=(0,0),\; C_1^t=\zeta_1,\; C_2^t=\zeta_1+\zeta_2,\; \dots ,\; C_d=C^t_d=\sum_{i=1}^d \zeta_i
\]
and $\Gamma_b$ is the curve obtained by connecting
\[
C_0=C_0^b,\; C_1^b=\zeta_{\sigma^{-1}(1)},\; C_2^b=\zeta_{\sigma^{-1}(1)}+\zeta_{\sigma^{-1}(2)},\;\dots,\; C_d=C_d^b.
\]

If $\zeta_1$ lies above $\zeta_{\sigma^{-1}(1)}$, i.e.\ if $\kappa_1>\kappa_{\sigma^{-1}(d)}$ then we call $\Gamma_t$ the top curve and $\Gamma_b$ the bottom curve. They have common end points $C_0$ and $C_d$. 
We denote their union by $\Gamma$. 
If there are no further intersections, $\Gamma$ bounds a polygon. 
In this case, after identifying the line segment $[C^t_{k-1},C^t_k]$ of $\Gamma_t$ with segment $[C^b_{j-1},C^b_{j}]$ in $\Gamma_b$ where $k=\sigma^{-1}(j)$.
One obtains a closed topological surface $S$ which outside of the corners of the polygon inherits a flat structure from $\R^2$. 
This means that there is an atlas of charts $\{(U,\psi)\}$ of $U$ open, and $\psi:U\to\R^2$ continuous such that for any two charts $\psi_i:U_i\to\C$ over a common point $p$, 
we find a translation $v\in \R^2$ such that $\psi_1=\psi_2+v$ for all points around $p$. It is possible to complement to an atlas defined on all of $M$ by considering the complex multiplication on $\R^2=\C$ with maps of the form $\psi=\phi^{\alpha+1}$ for a homeomorphism $\phi:U\to\C$ that maps a corner to $0$. The corners with $\alpha>0$ are called the singularities of $M$.

$M$ is endowed with a dynamical system, the vertical straight line flow that preserves the natural area form coming from the flat metric.
We note that the interval $I$ embeds into $M$ as a horizontal line starting from the origin.
The vertical straight line flow defines a suspension of the interval exchange transformation $T$ by considering the induced transformation on $I$, namely the first return map of $I\to I$.
We can now understand the meaning of $L(x)$: it is the return time of $x$ to $I$, and as such positive.

Self-intersections of the curves $\Gamma_t$ and $\Gamma_b$ give rise to a ``nonsensical picture'' (see depictions on page 247, \cite{minskyweiss}). We observe here that one can make sense of the picture even if the curve $\Gamma$ has self-intersection, by attaching a half-translation structure to it, but we have not tried to follow up on this direction.
Instead, we shall restrict ourselves to a criterion that avoids self-intersections.

\begin{lemma}
\label{lem:decreasing}
	Let ${\bf a}\in\R^d$ be a length vector, ${\bf b}\in\R^d$ be a height vector and $\Gamma_t:I\to\R^2$ be the top curve constructed by concatenating the vectors $\zeta_i=(a_i,b_i)\in \R^2$, i.e.\ $\Gamma_t(\sum_{i\leq j} a_i)=C^t_j$. Suppose the slopes $\kappa_i=\frac{b_i}{a_i}$ of $\zeta_i$ are strictly monotonically decreasing so that $\Gamma_t$ is convex. 
	Then for any irreducible permutation $\sigma$ and bottom curve $\Gamma_b$ constructed from vectors
	$\zeta_{\sigma^{-1}(1)}, \dots, \zeta_{\sigma^{-1}(d)}$, the closed curve $\Gamma=\Gamma_t\cup \Gamma_b$ has no self-intersections.
	In particular, $({\bf a}, {\bf b})$ defines a positive pair if connection-free.
\end{lemma}
\begin{proof}
We shall argue by induction on the number of symbols $d$.
The base case is on two elements $d=2$. 
By monotonicity $\kappa_1>\kappa_2$ and by irreducibility $\sigma=(2,1)$.
Then $\Gamma_t\cup\Gamma_b$ bounds a parallelogram.

Assume now that for all $d'<d$ the lemma is true. 
Let $\Gamma_{b,j}$ the curve from concatenating $C_0,C^b_1, \dots, C^b_j$ from left to right, i.e.\ restricting $\Gamma_b:I\to \R^2$ to $\cup_{i=1}^j I_{\sigma^{-1}(i)}$.
We now start another induction and assume that for all $j'<j$, $\Gamma_{b,j'}$ does not intersect $\Gamma_t$.
For the base of the induction $j=1$, there is nothing to check.

If $\Gamma_{b,j}$ intersects $\Gamma_t$ then by induction hypothesis 
it does so with its final line segment $[C^b_{j-1},C^b_{j}]$. We put $k=\sigma^{-1}(j)$ such that $C^b_{j-1}+\zeta_{k}=C^b_{j}$, intersecting, say, the $i$th segment $[C^t_{i-1},C^t_{i}]$ of $\Gamma_t$. Then $\kappa_{k}>\kappa_i$. 
By monotonicity, $\zeta_k$ has to appear to the left of $\zeta_i$ in $\Gamma_t$, i.e.\ $k<i$.

We now describe a procedure of removing $[C^b_{j-1},C^b_j]$ to obtain a smaller permutation to apply the induction hypothesis on $d'$.

Now observe that if $K\subset\sigma^{-1}(\{1,\dots,j\})$, 
we can define the curves $\Gamma_{t,K},\Gamma_{b,j,K}$ that one obtains from taking $\Gamma_{t}$ resp.\ $\Gamma_{b,j}$ 
and removing the line segments $[C^t_{k-1},C^t_{k}]$ for $k\in K$ from $\Gamma_{t}$ resp.\ $[C^b_{\sigma^{-1}(j'-1)},C^b_{\sigma^{-1}(j')}]$ from $\Gamma_{b,j}$ for $\sigma^{-1}(j')\in K$. 
Below, we shall have the additional property that $K\subset\{1,\dots,i-1\}$ for some $i$.
We obtain a new permutation $\sigma_K$ obtained by removing the symbols $k\in K$. 
If it is no longer irreducible then the maximal invariant subset $\{1,\dots,\ell\}$ 
must be contained in $\sigma^{-1}(\{1,\dots,j-1\})$
(or else $\sigma$ is already reducible).
By removing the sub-permutation on $(1\dots,\ell)$ from $\sigma_K$, 
we can allow ourselves to only consider the irreducible component $\sigma'$ of $\sigma_K$ containing $i$. 

We now choose $K=\{k'=\sigma^{-1}(j'): j'\leq j \text{ and } \kappa_{k'}\geq \kappa_{k} \}$.
We note that $k'\in K$ implies $k'<k=\sigma^{-1}(j)$ and that $i\not\in K$. 
Hence the curve $\Gamma_{t,K}$ is only changed to the left of its line segment $[C^t_{i-1},C^t_{i}]$,
and most importantly, the curve $\Gamma_{b,j,K}$ still intersects $[C^t_{i-1},C^t_{i}]$. To see this, divide the plane in two half-planes with boundary $\partial$ containing $\zeta_k$ attached to the right endpoint of $\Gamma_{b,j,K}$, and we see that $\Gamma_{b,j,K}$ stays to the upper-left half-plane. Since $[C^t_{i-1},C^t_{i}]$ intersects $\partial$, it also intersects $\Gamma_{b,j,K}$ as claimed.

If $\sigma_K$ is no longer irreducible then we proceed with the irreducible restriction $\sigma'$ as described above, supported on, say, $\cB\subset\cA$. Consider the associated to the pair $({\bf a'},{\bf b'})$ where
 ${\bf a'},{\bf b'}\in \R^{|\cB|}$ by restricting to the support of $\sigma'$. These give still monotone slopes, and the induction hypothesis on $d'<d$ applies, i.e.\ there are no self intersections. By construction, however, the curve defined by ${\bf a'},{\bf b'}$ and $\sigma'$ has at least one self-intersection.
\end{proof}

\begin{remark}
	We have an analogous criterion if $\kappa_{i}$ are increasing in $i$, in which case $\Gamma_b(\sum_{i\leq j} a_j)=C_j^t$, and we apply the argument of Lemma~\ref{lem:decreasing} with roles of $\Gamma_t$ and $\Gamma_b$ exchanged.
\end{remark}
\begin{remark}
	Barak Weiss has informed us on a topological proof of Lemma~\ref{lem:decreasing} which we invite the reader to find herself.
\end{remark}
\begin{corollary}
	Theorem~\ref{thm:mw-mahler} holds for any irreducible permutation.
\end{corollary}
\begin{proof}
	The slopes associated to $\textbf{a}(s)=(s,s^2,\dots,s^d)$ are $\kappa_i=\frac{is^{i-1}}{s^i}=\frac{i}{s}$, monotone in $i$.
\end{proof}

\bibliographystyle{siam}

\end{document}